\documentclass{amsart}
\usepackage[utf8]{inputenc}
\usepackage{verbatim}
\usepackage[backend=bibtex,style=alphabetic,doi=false,isbn=false,url=false,giveninits=true,maxbibnames=50]{biblatex}
\usepackage{enumitem}
\setenumerate{label=(\roman*)}

\newcommand{\nc}{\newcommand}
\nc{\lie}[1]{\mathfrak{#1}}

\theoremstyle{plain}
\newtheorem{thm}[subsection]{Theorem}
\newtheorem*{thm*}{Theorem}
\theoremstyle{definition} 
\newtheorem{lem}[subsection]{Lemma}
\newtheorem{cor}[subsection]{Corollary}

\newtheorem{prop}[subsection]{Proposition}

\newtheorem{defn}[subsection]{Definition}

\nc{\C}{\mathbb{C}}

\title{A minimal Gr\"obner basis for simple $\mathfrak{sl}_n$- or $\mathfrak{sp}_n$-modules}
\author[]{Ghislain Fourier and León van Eß}
\address{Chair of Algebra and Representation Theory, RWTH Aachen University}
\email{fourier@art.rwth-aachen.de}
\email{leon.van.ess@rwth-aachen.de}

\begin{document}

\begin{abstract}
We explicitly provide minimal Gröbner bases for simple, finite-dimensional modules of complex Lie algebras of types A and C, using a homogeneous ordering that is compatible with the PBW filtration on the universal enveloping algebras.
\end{abstract}
\maketitle
\section{Introduction}
For a Lie algebra, the PBW (Poincaré-Birkhoff-Witt) filtration on its enveloping algebra is induced by the natural degree. The famous PBW theorem states that the associated graded algebra is the symmetric algebra. PBW filtrations and degenerations of simple, finite-dimensional, complex Lie algebras $\lie g$ and their finite-dimensional modules have been in focus in the past fifteen years, due to their geometric implications (see, for example, \cite{Fei12}) and their related combinatorics (\cite{BF20, FR21}). In \cite{FFLa1, FFLa2}, new monomial bases (sometimes called FFLV bases) of these modules that are compatible with the PBW filtration have been provided.\\
Simple, finite-dimensional modules for simple, finite-dimensional complex Lie algebras are indexed by their highest weight $\lambda$. Then $V(\lambda) \cong U(\lie g)/I_\lambda$ for a left ideal $I_\lambda$. In this note, we are considering a monomial ordering on $U(\lie g)$ introduced in \cite{FFLa1, FFLa2} and compute an explicit (left-) Gröbner basis $M_\lambda$ (see Definition~\ref{defn-gb})  for $I_\lambda$. Since we are using the results of loc. cit., we restrict ourselves to Lie algebras of type $A$ and $C$. The main theorem is
\begin{thm*}
Let $\lie g$ be of type $A$ or $C$, and $\lambda$ a dominant integral weight. $M_\lambda$ is a minimal (left) Gröbner basis of $I_\lambda$ and the induced monomial basis of $V(\lambda)$ is the FFLV basis. 
\end{thm*}
The first non-trivial example is for $\lie{sl}_3$, and $\lambda$ being the highest root (or the minimal regular, dominant, integral weight). A minimal Gröbner basis is then (the simple roots are $\alpha_1, \alpha_2$) for the particular ordering (see Section~\ref{subsec-gb}):
\begin{multline*}
\{ e_{\alpha_1}, e_{\alpha_2}, e_{\alpha_1 + \alpha_2}, h_{\alpha_1} - 1, h_{\alpha_2} - 1,\\  f_{\alpha_1}^2, f_{\alpha_2}^2, f_{\alpha_1 + \alpha_2}^3, f_{\alpha_2}f_{\alpha_1 + \alpha_2} f_{\alpha_1} + \frac{1}{2} f_{\alpha_1 + \alpha_2}^2,  f_{\alpha_1 + \alpha_2}^2 f_{\alpha_1},  f_{\alpha_2} f_{\alpha_1 + \alpha_2}^2 \}
\end{multline*}

\medskip
Our proofs rely on the straightening law proved in loc. cit. Although FFLV type bases are known for $G_2$ (\cite{Gor15}) and in type $B$ (\cite{Mak19}), our methods cannot be applied right away as they do not provide a similar straightening law.\\
Note that Gröbner bases of $I_\lambda$ have been computed in the context of Gröbner-Shirshov bases for different orderings, for example in \cite{KL00}. Their monomial basis provides a Gelfand-Tsetlin type monomial basis \cite{GT50} instead of the FFLV basis.\\
Note further that $M_\lambda$ is generally not reduced. Consider $\lambda = \omega_1 + \omega_2 + 2 \omega_3$ for $\lie{sl}_4$; then one has
\[
     f_{\alpha_2+\alpha_3}^2 f_{\alpha_1+\alpha_2+\alpha_3} f_{\alpha_1+\alpha_2}^2 + 2  f_{\alpha_2+\alpha_3} f_{\alpha_2} f_{\alpha_1+\alpha_2+\alpha_3}^2 f_{\alpha_1+\alpha_2} + \frac{1}{3} f_{\alpha_2}^2  f_{\alpha_1+\alpha_2+\alpha_3}^3
\]\\
as an element of $M_\lambda$. Since $f_{\alpha_2}^2$ is also contained, the set is not reduced.
These computations are carried out using OSCAR \cite{Oscar} and the implementation of the Gröbner basis for other types can be found in the latter.
\medskip

The paper is organized as follows: In Section~\ref{sec-pre}, we recall the PBW filtration and the FFLV polytopes. In Section~\ref{sec-main}, we provide the proof of the main result.
\medskip

\noindent\textbf{Acknowledgments} 
The first author gratefully acknowledges financial support by the DFG –Project-ID 286237555–TRR 195. The authors would like to express their gratitude to Thomas Breuer for his patient explanations of GAP.

\section{Preliminaries}\label{sec-pre}
Let $\lie g$ be a finite-dimensional, complex simple Lie algebra and $ \lie g = \lie n^+ \oplus \lie h \oplus \lie n^-$ be a triangular decomposition. We denote the set of roots (positive roots resp.) $R$ ($R^+$), the simple roots are denoted $\alpha_1, \ldots, \alpha_n$, the weight lattice $P$, the dominant weights $P^+$. 
The fundamental weights are denoted $\omega_1, \ldots, \omega_n$. For each $\beta \in R^+$, we fix root operators $e_\beta \in \lie n^+_{\beta}$ and $f_\beta \in \lie n^-_{- \beta}$. The universal enveloping algebra is denoted $U(\lie g)$.\\
For $\lambda \in P^+$, we denote $V(\lambda)$ the simple, finite-dimensional highest weight module of highest weight $\lambda$ and fix a highest weight vector $v_\lambda$, then $V(\lambda) = U(\lie n^-).v_\lambda$. We denote $I_\lambda \subset U(\lie g)$ the defining left ideal of $V(\lambda)$ as a module for $U(\lie g)$.

\subsection{PBW filtration}
For a given Lie algebra $\lie a$, there is a natural filtration on the universal enveloping algebra $U(\lie a)$:
\[
U(\lie a)_s := \langle x_{i_1} \cdots x_{i_t} \mid t \leq s, x_{i_j} \in \lie a \rangle_\C
\]
The PBW theorem states, that the associated graded algebra is isomorphic to the symmetric algebra on (the vector space) $\lie a$. Consequently, for every cyclic $\lie a$-module $M$ with fixed generator $m$, we obtain an induced $S(\lie a)$-module $M^a$.\\

We apply this construction to $\lie n^-$ and the cyclic module $V(\lambda)$ with generator $v_\lambda$. In this case, we are not only obtaining a $S(\lie n^-)$-module structure on $V(\lambda)^a$ but, since the PBW filtration on $\lie n^-$ is invariant under the $\lie b := \lie h \oplus \lie n^+$-action, an action of $S(\lie n^-)U(\lie b)$. This is called the PBW degenerate module $V(\lambda)^a$, we denote $I(\lambda)^a \subset S(\lie n^-)$ the defining ideal. \\
We recall from \cite{FFLa1, FFLa2} the notion of Dyck paths for Lie algebras of type $A$ and $C$. Here write short $\alpha_{i,j} = \alpha_i + \ldots + \alpha_j$ and $\alpha_{i, \overline{j}} = \alpha_i + \ldots + \alpha_n + \ldots + \alpha_{j}$ for $i \leq j$. We introduce an ordering on positive roots 
\[
\alpha_{i,j} \leq \alpha_{k,\ell} :\Leftrightarrow i \leq j \text{ and } j \leq \ell
\] with $1 < 2 < \ldots < n < \overline{n-1} < \ldots < \overline{1}$.
Let $\lie g$ be of type $A$ or $C$, a sequence of positive roots $(\beta_1, \ldots, \beta_s)$ is called a Dyck path if
\begin{itemize}
    \item $\beta_1$ is a simple root.
    \item $\beta_s$ is either a simple root or highest root of a type $C$ root system.
    \item $\beta_k < \beta_{k+1}$ for all $k$.
\end{itemize}
The set of Dyck paths starting in $\alpha_i$ and ending in $\alpha_j$ are denoted $D_{i,j}$, the Dyck paths starting in $\alpha_i$ and ending in any highest root of a type $C$ subdiagram are denoted $D_{i,n}$.
The following has been shown in \cite{FFLa1, FFLa2}
\begin{thm}
    Let $\lambda = \sum m_i \omega_i \in P^+$, then 
    \[
    \left\{\left. \prod_{\alpha \in R^+} f_\alpha^{s_\alpha}.v_\lambda \in V(\lambda) \;\right|\; \forall \mathbf{p} \in D_{i,j}: \sum_{\alpha \in \mathbf{p}} s_\alpha \leq m_i + \ldots + m_j \right\}
    \]
    forms a basis of $V^a(\lambda)$ and for any fixed ordering in the monomials also a basis of $V(\lambda)$.
\end{thm}
These bases were constructed to provide generators of the ideal $I^a(\lambda)$, which turns out to be  \cite{FFLa1, FFLa2}
$$\{ U(\lie n^+).f_\alpha^{\lambda(h_\alpha)+1} \}. $$\\

\subsection{Gröbner bases}\label{subsec-gb}
For the following, we need to introduce a bit of notation:
Let $\lambda  = \sum m_i \omega_i\in P^+$, we set $\lambda_i = \sum_{j=i}^n m_j$. Then $\lambda$ is also the partition $(\lambda_1 \geq \lambda_2 \geq \ldots \geq \lambda_n \geq 0)$. Let $\mathbf{s} \in\mathbb{R}^{|R^+|}$ be supported on the roots of the subdiagram induced from $\alpha_\ell, \ldots, \alpha_k$. \\

Since $I_\lambda$ is finitely generated and  $U(\lie g)$ is a G-algebra (also known as PBW algebra or algebra of solvable type), a (left-)Gröbner basis exists for each adequate monomial ordering (see \cite{Kan}). We recall the monomial ordering on $U(\lie n^-)$ introduced in \cite{FFL17} here:\\
A refinement of the natural ordering on positive roots is called a good ordering, for example $\alpha_{1,3} > \alpha_{1,2} > \alpha_{2,3} > \alpha_1 > \alpha_2 > \alpha_3$ for $\lie{sl}_4$. We consider the induced degree reverse lexicographic ordering on $U(\lie n^-)$ (all root vectors have degree $1$ and $(1,0,0,1) < (0,1,1,0)$) and extend this to $U(\lie g)$ by setting $\lie n^- > \lie h > \lie n^+$. \\

For type $A$ we set
\[
s_{\bullet, j} = \sum_{i = \ell}^j s_{i,j}, \; \; \;  s_{i, \bullet} = \sum_{j=i}^k s_{i,j} 
\] 
to define (following again \cite{FFLa1})
\[
\partial_{\mathbf{s}} \; - := [e_{\ell, \ell}^{s_{\ell+1, \bullet}}, [\;\cdots, [e_{\ell,k-1}^{s_{k, \bullet}}, [e_{k,k}^{s_{\bullet, k-1}}, [\;\cdots, [ e_{\ell+1,k}^{s_{\bullet, \ell}}, -] \cdots ]]] \cdots ]].
\]
To ease the notation, we do not recall $\partial_{\mathbf{s}}$ in type $C$ but refer to \cite[Theorem 3.4]{FFLa2}.

\medskip

Let $\mathbf{s}$ be supported on a Dyck path $\mathbf{p} \in D_{i,j}$, then it is shown in \cite{FFLa1, FFLa2}, that $LM(\partial_{\mathbf{s}} f_{i,j}^{\text{deg } \mathbf{s}}) = f^{\mathbf{s}}$. 
\medskip

\begin{defn}\label{defn-gb}
\noindent We define the Gröbner basis $M_\lambda$ as the union of $\{e_\alpha, h_\alpha - \lambda(h_\alpha) \mid \alpha \in R^+\}$ and
\[
\bigcup_{1 \leq i \leq j \leq n} \left\{ \partial_{\mathbf s} f_{ij}^{\deg \mathbf s} \left| \; \begin{matrix}
                \mathbf s \text{ is supported on } \mathbf p \in D_{ij}, \deg \mathbf s = |\lambda_j-\lambda_i|+1 \text{ and} \\
                \forall \mathbf t\ne \mathbf s, \mathbf t \leq \mathbf s, \mathbf t \text{ supported on }\mathbf q \in D_{kl}: \deg \mathbf t \leq |\lambda_l-\lambda_k|
            \end{matrix}\right.\right\}
\]
Here $ \mathbf t \leq \mathbf s$ compares componentwise. 
\end{defn}
The main result of this paper is the following
\begin{thm}\label{thm:main}
$M_\lambda$ is a minimal Gröbner basis of $I(\lambda)$ with respect to the ordering $\geq$.
\end{thm}

\begin{cor}
If $K^a(\lambda)$ is the left ideal in $S(\lie n^-)U(\lie b)$ generated by $I^a(\lambda)$, then $M_\lambda$ is also a minimal Gröbner basis of $K^a(\lambda)$.
\end{cor}

\begin{proof}
The PBW degeneration is compatible with the monomial ordering.
\end{proof}

\section{Proof of the main result}\label{sec-main}
The following proposition is a standard result deduced from the PBW theorem.
\begin{prop}
    Let $f_{\beta_1}^{s_1} \cdots f_{\beta_k}^{s_k}$ be an ordered product in $U(\lie n^-)_s$. For each permutation $\sigma$, there exists $\mathbf{m} \in U(\lie n^-)_{s-1}$, such that
    \[
        f_{\beta_1}^{s_1} \cdots f_{\beta_k}^{s_k} = f_{\beta_{\sigma(1)}}^{s_{\sigma(1)}} \cdots f_{\beta_{\sigma(k)}}^{s_{\sigma(k)}} + \mathbf{m}.
    \]
    with $\text{deg } \mathbf{m}  < \sum s_i$.
\end{prop}

The proof of Theorem~\ref{thm:main} consists of three steps as we outline here. We denote $J_\lambda$ the ideal generated by the set $M_\lambda$. 
\begin{enumerate}
    \item[Step 1] Prove that $I_\lambda = J_\lambda$ 
    \item[Step 2] Prove that $M_\lambda$ is a Gröbner basis.
    \item[Step 3] Prove that $M_\lambda$ is a minimal Gröbner basis.
\end{enumerate}
\begin{proof}[Step 1]
$I_\lambda$ as left $U(\lie g)$ ideal is generated by (\cite[Theorem 21.4]{Hum78})
$$\{e_\alpha, h_\alpha - \lambda(h_\alpha), f_\alpha^{\lambda(h_\alpha) + 1} \mid \alpha \in R^+ \}$$
These elements are contained in $M_\lambda$ and hence $I_\lambda \subseteq J_\lambda$. 
On the other hand, let $\mathbf{s} \in \mathbb{R}^{|R^+|}$ be supported on a Dyck path $\mathbf{p} \in D_{i,j}$ and $\text{deg } \mathbf s = | \lambda_i - \lambda_j| +1$, then $f_{i,j}^{\text{deg} \mathbf s} = 0$ and hence $\partial_{\mathbf{s}} f_{i,j}^{\text{deg } \mathbf s} \in I_\lambda$. 
Which implies $M_\lambda \subseteq I_\lambda$ and so $J_\lambda \subseteq I_\lambda$.
\end{proof}

\begin{proof}[Step 2]
By definition of a Gröbner basis, we have to show that 
$$
LT(J_\lambda) = ( LT(f) \mid f \in M_\lambda) =: (LT(M_\lambda))
$$
Let $m \in I_\lambda$, we order each term of $m$ with respect to $\prec$. Then the leading term of $m$ is of the form
$$
(\prod_\alpha f_\alpha^{s_\alpha})(\prod_\alpha h_\alpha^{t_\alpha})(\prod_\alpha e_\alpha^{u_\alpha})
$$
If there is $\alpha \in R^+$ with $u_\alpha \neq 0$, then $m \in U(\lie g).e_\alpha \subseteq (LT(M_\lambda))$. If $u_\alpha = 0$ for all $\alpha \in R^+$ and there exists $\alpha \in R^+$ with $t_\alpha \neq 0$. Using $h_\alpha - \lambda(h_\alpha) \in M_\lambda$, showing that $LT(m) \in (LT(M_\lambda))$ reduces to showing that $LT(m') \in (LT(M_\lambda))$ for all $m'$ such that $LT(m') \in LT(I(\lambda)) \cap U(\lie n^-)$, that is $t_\alpha = u_\alpha = 0$ for all $\alpha \in R^+$.\\
Suppose now $m \in I(\lambda) \cap U(\lie n^-)$ with leading term $LT(m) = \prod_\alpha f_\alpha^{s_\alpha}$. Suppose there is no Dyck path $\mathbf{p} \in D_{i,j}$ such that $\sum_{\alpha \in  \mathbf{p}} s_\alpha > \lambda_i - \lambda_j$, then $LT(m)$ is in the FFLV basis of $V(\lambda)$ and since $m = LT(m) + \sum m'$, with $m' < LT(m)$, $ m\notin I_\lambda$. So there exists  $\mathbf{p} \in D_{i,j}$ with  $\sum_{\alpha \in  \mathbf{p}} s_\alpha > \lambda_i - \lambda_j$. We set 
$$\mathbf{s}' :=  \begin{cases} s_\alpha &\mid \alpha \in \mathbf{p} \\ 0 &\mid \text{ else } \end{cases} $$
We can assume that $\mathbf{p}$ is a minimal (by length) Dyck path $\mathbf{p}$ on which $LT(m)$ violates the FFLV conditions and adjust $\mathbf{s}'$ such that $\text{deg } \mathbf{s}' = \lambda_i - \lambda_j + 1$. 
Then $\partial_{\mathbf{s}'} f_{i,j}^{\text{deg } \mathbf{s}'} \in M_\lambda$ and $LT(m)$ is a multiple of $LT(\partial_{\mathbf{s}'} f_{i,j}^{\text{deg } \mathbf{s}'})$.
\end{proof}
\begin{proof}[Step 3]
Let $m \in M_\lambda$ and suppose $LT(m) \in (LT(M_\lambda \setminus\{m\}))$. Then there exists $m'$ with $LT(m') \in LT(M_\lambda \setminus\{m\})$ such that $LT(m)$ is a multiple of $LT(m')$. 
Let $\mathbf{s}, \mathbf{t}$ such that $LT(m') = LT(\partial_{\mathbf{s}} f_{i,j}^{\text{deg } \mathbf{s}})$ and $LT(m) = LT(\partial_{\mathbf{t}} f_{k, \ell}^{\text{deg } \mathbf{t}})$. Let $\mathbf{q}$ (resp. $\mathbf{p}$) be the Dyck paths whose conditions are violated (following the definition of $M_\lambda$).\\
Then $\text{supp } \mathbf{s} \subset \mathbf{q}, \text{supp } \mathbf{t} \subset \mathbf{p}$. Since $LT(m')$ divides $LT(m)$, one has $\text{supp } \mathbf{s} \subset \text{supp } \mathbf{t}$. 
One can adjust $\mathbf{p}$ on the complement of $\text{supp } \mathbf{t}$ such that $\text{supp } \mathbf{s}  \cap \mathbf{q} \subset \mathbf{p}$. If $\mathbf{p} \neq  \mathbf{q}$, this would imply that $m \notin M_\lambda$ since $\mathbf{s} | \mathbf{t}$. \\
Hence, the total degrees of $m$ and $m'$ are equal and accordingly, so $LT(m) = LT(m')$ which is not possible due to the construction of $\partial_{\mathbf{s}} f_{i,j}^{\text{deg } \mathbf{s}}$.
\end{proof}

\printbibliography
\end{document}